\pdfoutput=1 

\documentclass[11pt]{amsart}

\usepackage{amssymb,amsthm,enumitem,colonequals,tikz-cd,microtype}
\usepackage[normalem]{ulem} 

\usepackage[osf]{Baskervaldx}
\usepackage[baskervaldx]{newtxmath}
\usepackage[cal=boondoxo]{mathalfa} 

\usepackage[top=3.75cm, bottom=3cm, left=3.5cm, right=3.5cm]{geometry}

\usepackage{xcolor}
\colorlet{darkblue}{blue!55!black}
\colorlet{darkcyan}{cyan!50!black}
\colorlet{darkgreen}{green!60!black}

\PassOptionsToPackage{hyphens}{url}
\usepackage{hyperref}
\hypersetup{
    colorlinks=true,
    linkcolor=darkblue,
    urlcolor=darkcyan,
    citecolor=darkgreen,
}

\def\eqref#1{\textcolor{darkblue}{(\ref{#1})}}

\usepackage[nameinlink]{cleveref} 
\Crefformat{section}{#2\S#1#3}
\Crefmultiformat{section}{#2\S\S#1#3}{ and~#2#1#3}{, #2#1#3}{, and~#2#1#3}

\usepackage[pagewise]{lineno}
\let\oldequation\equation
\let\oldendequation\endequation
\renewenvironment{equation}{\linenomathNonumbers\oldequation}{\oldendequation\endlinenomath}
\expandafter\let\expandafter\oldequationstar\csname equation*\endcsname
\expandafter\let\expandafter\oldendequationstar\csname endequation*\endcsname
\renewenvironment{equation*}{\linenomathNonumbers\oldequationstar}{\oldendequationstar\endlinenomath}
\let\oldalign\align
\let\oldendalign\endalign
\renewenvironment{align}{\linenomathNonumbers\oldalign}{\oldendalign\endlinenomath}
\expandafter\let\expandafter\oldalignstar\csname align*\endcsname
\expandafter\let\expandafter\oldendalignstar\csname endalign*\endcsname
\renewenvironment{align*}{\linenomathNonumbers\oldalignstar}{\oldendalignstar\endlinenomath}

\theoremstyle{plain}
\newtheorem{theorem}{Theorem}[section]
\newtheorem{lemma}[theorem]{Lemma}
\newtheorem{corollary}[theorem]{Corollary}
\newtheorem{proposition}[theorem]{Proposition}

\theoremstyle{definition}
\newtheorem{definition}[theorem]{Definition}
\newtheorem{example}[theorem]{Example}
\newtheorem{remark}[theorem]{Remark}

\newtheorem*{ack}{Acknowledgments}

\AddToHook{env/conjecture/begin}{\crefalias{theorem}{conjecture}}
\AddToHook{env/lemma/begin}{\crefalias{theorem}{lemma}}
\AddToHook{env/corollary/begin}{\crefalias{theorem}{corollary}}
\AddToHook{env/proposition/begin}{\crefalias{theorem}{proposition}}
\AddToHook{env/definition/begin}{\crefalias{theorem}{definition}}
\AddToHook{env/remark/begin}{\crefalias{theorem}{remark}}
\AddToHook{env/setup/begin}{\crefalias{theorem}{setup}}
\AddToHook{env/example/begin}{\crefalias{theorem}{example}}
\AddToHook{env/conjecture/begin}{\crefalias{theorem}{Conjecture}}

\numberwithin{equation}{section}
\numberwithin{theorem}{section}

\title{Measuring birational derived splinters}

\author[T.~De Deyn]{Timothy De Deyn}
\address{T.~De Deyn,
Max Planck Institute for Mathematics,
Bonn, Germany}
\email{dedeyn@mpim-bonn.mpg.de}

\author[P.~Lank]{Pat Lank}
\address{P.~Lank,
Dipartimento di Matematica “F. Enriques”, Universit\`{a} degli Studi di Milano, Milano, Italy}
\email{plankmathematics@gmail.com}

\author[K.~Manali Rahul]{Kabeer Manali Rahul}
\address{K.~Manali Rahul,
Max Planck Institute for Mathematics,
Bonn, Germany}
\email{kabeermr.maths@gmail.com}

\author[S.~Venkatesh]{Sridhar Venkatesh}
\address{S.~Venkatesh,
Department of Mathematics,
UCLA, 
Los Angeles, CA, U.S.A,}
\email{srivenk@math.ucla.edu}

\date{\today}

\keywords{(birational) derived splinters, measurements, generation for triangulated categories}

\subjclass[2020]{14A30 (primary), 14F08, 14B05, 18G80} 

\begin{document}
    
\begin{abstract}
    This work is concerned with categorical methods for studying singularities. Our focus is on birational derived splinters, which is a notion that extends the definition of rational singularities beyond varieties over fields of characteristic zero. Particularly, we show that an invariant called `level' in the associated derived category measures the failure of these singularities.
\end{abstract}

\maketitle


\section{Introduction}
\label{sec:intro}

Nowadays, it is becoming more and more clear that many geometric properties of a variety are reflected by the behavior of its associated derived category. In fact, many types of singularities fall under such geometric properties. Recall that a variety $X$ is a \textit{splinter} if the natural morphism $\mathcal{O}_X \to f_\ast \mathcal{O}_Y$ splits for all finite surjective morphisms $f\colon Y \to X$. This term was coined in \cite{Sing:1999}, although the idea dates back further. 

An inspiration for related notions is Hochster's \textit{direct summand conjecture} \cite{Hochster:1973}. It asserts that any affine regular scheme is a splinter. Recently, Andr\'{e} proved Hochster's conjecture using perfectoid spaces \cite{Yves:2018}, and shortly after, Bhatt proved a derived variation \cite{Bhatt:2018}. More history on the conjecture can be found in \cite{Hochster:2007}.

The behavior of splinters can vary depending on the characteristic of $X$. In the case $X$ has characteristic zero, being a splinter is equivalent to being a normal scheme, which implies the singularity is a local property.  However, in positive characteristic, the story is different. 
Specifically, Bhatt showed $X/\mathbb{F}_p$ is a splinter if and only if it is a \textit{derived splinter}; that is, the natural morphism $\mathcal{O}_X \to \mathbf{R} f_\ast \mathcal{O}_Y$ splits for every proper surjective morphism $f\colon Y \to X$. It turns out that being a splinter is not a local property in such cases. Additionally, it is worthwhile to note that in characteristic zero, being a derived splinter is the same as having rational singularities \cite{Kovacs:2000,Bhatt:2012}.

Part of Bhatt's intuition was the ansatz that proper morphisms are robust derived analogues of finite morphisms. Motivated from this perspective, we study how the addition of birationality to the proper morphisms influences the behavior of such singularities and how this is reflected by their derived categories. Combining these ideas draws attention to the following notion first introduced by Kov\'{a}cs.

\begin{definition}\label{def:bds}
    A Noetherian scheme $X$ is called a \textit{birational derived splinter} if the natural morphism $\mathcal{O}_X \to \mathbf{R}f_\ast \mathcal{O}_Y$ splits for every proper birational morphism $f\colon Y \to X$.
\end{definition}

These singularities have recently been gaining traction, e.g.\ \cite{Lyu:2022,Ma/McDonald/RG/Schwede:2025}.
Furthermore, they coincide with derived splinters in several cases, including affine regular schemes \cite{Bhatt:2018} or affine quasi-excellent $\mathbb{Q}$-schemes \cite[Proposition 3.5]{Liu:2002} (which we generalize to quasi-excellent $\mathbb{Q}$-schemes, see \Cref{cor:birational_notions_for_quasi_excellent_char_zero}). However, the situation changes drastically for schemes in positive characteristic. In particular, any elliptic curve in positive characteristic is a birational derived splinter, whereas Bhatt showed that it cannot be a derived splinter \cite[Example 2.11]{Bhatt:2012}.

One important aspect of the notions of splinters, derived splinters, and birational derived splinters is that they can be studied in any characteristic: zero, positive, or mixed. This is in contrast to rational singularities which require certain vanishing results to behave well, and these vanishing results need not always hold outside of characteristic zero. Moreover, the splinter singularities are related to various other interesting singularities. In positive characteristic, weakly $F$-regular rings are splinters, whereas splinters are $F$-rational \cite{Smith:1993}. In mixed characteristic, $\operatorname{BCM}$-rational or $+$-rational singularities are birational derived splinters (see \cite{Ma/Schwede:2021} and \cite[Remark 3.9]{Ma/McDonald/RG/Schwede:2025}). Furthermore, $\operatorname{BCM}$-rational or $+$-rational singularities are pseudo-rational \cite{Ma/Schwede:2021}, and pseudo-rational schemes with a canonical sheaf are birational derived splinters \cite[Lemma 3.8]{Ma/McDonald/RG/Schwede:2025}.

\subsection{Measurements}
\label{sec:intro_measurement}

Quantifying singularities is important because this can help understand the failure of the singularity existing. In \cite{Lank/Venkatesh:2025}, the notion of level was used to characterize the failure of being a derived splinter. Motivated by this, we propose a numerical invariant (of the scheme) that measures the failure of being a birational derived splinter. Unlike \textit{loc.\ cit.}, we additionally study how geometric properties influence the behavior of this invariant which is possible due to the birationality of the morphisms defining birational derived splinters.

Our techniques use a notion of generation in a triangulated category $\mathcal{T}$ introduced in \cite{Bondal/VandenBergh:2003} and its associated invariants from \cite{Avramov/Buchweitz/Iyengar/Miller:2010}. For a subcategory $\mathcal{S}\subseteq \mathcal{T}$, let $\langle \mathcal{S} \rangle_{n+1}$ (for $n \geq 0$) denote the subcategory generated by $\mathcal{S}$ using shifts, direct sums, direct summands, and at most $n$ cones. Given $E \in \mathcal{T}$, the \textit{level} of $E$ with respect to $\mathcal{S}$, denoted $\operatorname{level}^\mathcal{S}(E)$, is the smallest $N$ such that $E \in \langle \mathcal{S} \rangle_N$, or $+\infty$ otherwise. See \Cref{sec:prelim_generation} for details.

To be a birational derived splinter involves splitting conditions on the `natural morphism' between structure sheaves in the derived category. Before defining our numerical invariant, we motivate them with following lemma; it shows the relation between the splitting conditions and generation in the triangulated category.

Recall the bounded derived category of coherent sheaves is denoted by $D^b_{\operatorname{coh}}$.

\begin{lemma}
    \label{lem:thick_one_iff_splitting_naturally}
    Let $f\colon Y \to X$ be a proper surjective morphism of Noetherian schemes. Then the natural morphism $\mathcal{O}_X \to \mathbf{R}f_\ast \mathcal{O}_Y$ splits if, and only if, $\operatorname{level}^{\mathbf{R}f_\ast D^b_{\operatorname{coh}}(Y) } (\mathcal{O}_X) = 1$.
\end{lemma}

\begin{proof}
    The natural morphism $\mathcal{O}_X \to \mathbf{R}f_\ast \mathcal{O}_Y$ splitting implies the claim on level, whereas the converse direction follows via adjunction. This can be seen using e.g.\ use \cite[Lemma 5.15]{DeDeyn/Lank/ManaliRahul:2024b} but we spell out a proof. Indeed, the hypothesis tells us $\mathcal{O}_X$ is a direct summand of a bounded complex of the form $\oplus_{n \in \mathbb{Z}} \mathbf{R}f_\ast E_n^{\oplus r_n} [n]$ where the coproduct is finite and $E_n\in D^b_{\operatorname{coh}}(Y)$. We know that $\oplus_{n \in \mathbb{Z}} \mathbf{R}f_\ast E_n^{\oplus r_n} [n] \cong \mathbf{R}f_\ast (\oplus_{n \in \mathbb{Z}} E_n^{\oplus r_n} [n])$. Hence, it follows that there is a morphism $\mathcal{O}_X \to \mathbf{R}f_\ast (\oplus_{n \in \mathbb{Z}} E_n^{\oplus r_n} [n])$ which has a left inverse. However, using the push/pull adjunction, this morphism factors through the natural morphism $\mathcal{O}_X \to \mathbf{R}f_\ast \mathcal{O}_Y$ (which is just the unit of the adjunction). Consequently, we found the desired splitting.
\end{proof}

Now, the stage has been set.

\begin{definition}
    Let $X$ be a Noetherian scheme.
    Define $\mu_{\operatorname{bds}}(X)$ to be the supremum over the levels of $\mathcal{O}_X$ with respect to $\mathbf{R}f_\ast D^b_{\operatorname{coh}}(Y)$ for every  proper and birational morphism $f\colon Y \to X$. 
    If $X=\operatorname{Spec}(R)$ we also write $\mu_{\operatorname{bds}}(R)$ instead of $\mu_{\operatorname{bds}}(X)$. In symbols,
    \begin{displaymath}
        \mu_{\operatorname{bds}}(X) := \sup\left\{ \operatorname{level}^{\mathbf{R}f_\ast D^b_{\operatorname{coh}}(Y) } (\mathcal{O}_X) \mid f\colon Y \to X \textrm{ proper and birational} \right\}.
    \end{displaymath}
\end{definition}

Philosophically, $\mu_{\operatorname{bds}}$ measures the failure of $X$ to be a birational derived splinter. In particular, this singularity occurs precisely when $\mu_{\operatorname{bds}}(X)=1$. Moreover, we show that $\mu_{\operatorname{bds}}(X)$ can be characterized by blowups along nonzero ideal sheaves (see \Cref{prop:ds_by_alteration}). 

First, we focus on studying $\mu_{\operatorname{bds}}(X)$ in the presence of a resolution (of singularities) $f\colon \widetilde{X} \to X$. In this setting, $\mu_{\operatorname{bds}}(X)$ is always finite and is independent of the chosen $f$ (see \Cref{thm:r_value_for_modification_finite}). Furthermore, $\mu_{\operatorname{bds}}$ is invariant for projective bundles (see \Cref{cor:bds_bundles}). 

Next, we consider the affine case; again in the setting where resolution of singularities exist. For $(R,\mathfrak{m})$ a quasi-excellent normal local ring with $\mathfrak{m}$-adic completion $\widehat{R}$, we have $\mu_{\operatorname{bds}}(R)\geq \mu_{\operatorname{bds}}(\widehat{R})$ (see \Cref{cor:completions_bds}). Moreover, if $X$ is an integral normal affine scheme, we show (see \Cref{prop:measurement_for_affine_local_to_global})
\begin{displaymath}
    \mu_{\operatorname{bds}}(X) = \sup_{p \in X} \{ \mu_{\operatorname{bds}}(\mathcal{O}_{X,p}) \}.
\end{displaymath}

\subsection{Some consequences}
\label{sec:intro_consequences}

Throughout the subsection, we work with schemes of finite type over a field $k$. Note that the results of this section require the existence of a resolution of singularities. In positive characteristic, it is not immediately clear whether the birational derived splinter property is preserved or reflected by fiber products or base change. Key tools commonly used in characteristic zero, such as Kov\'{a}cs' splitting criteria \cite{Kovacs:2000} or Grauert--Riemenschneider vanishing, are generally not available. This complicates studying these singularities in positive characteristic geometry.

Yet, in the affine setting, birational derived splinters behave well with respect to geometric operations, e.g.\ they are stable under localizations, direct limits, and base change along \'{e}tale extensions (see \cite{Lyu:2022}). Furthermore, when the scheme is Cohen--Macaulay, being a birational derived splinter coincides with being pseudorational \cite{Lipman/Teissier:1981}, showing that in this case they are stalk local (see \Cref{lem:pseudorational_bds_same}). However, being a birational derived splinter need not imply Cohen--Macaulayness, see \cite[Example 8.6]{Ma/Polstra:2025} and \cite[Remark 3.10]{Ma/McDonald/RG/Schwede:2025}. 
This makes studying birational derived splinters in the global setting a bit more complicated in the absence of Cohen--Macaulayness. 
We show that $\mu_{\operatorname{bds}}$, and hence generation for triangulated categories, can help clarify the picture. To start, we prove the following inequalities that describe the behavior of $\mu_{\operatorname{bds}}$ under these operations.

\begin{theorem}
    [see \Cref{prop:independence_field,prop:submultiplicativity}]
    \label{introthm}
    Let $k$ be perfect. Consider proper $k$-schemes $Y_1$ and $Y_2$ which admit resolutions of singularities. Then
    \begin{enumerate}
        \item $\mu_{\operatorname{bds}}(Y_1) \geq \mu_{\operatorname{bds}}(Y_1\times_k \operatorname{Spec}(L))$ for every field extension $L/k$,
        \item $\max\{\mu_{\operatorname{bds}}(Y_1),\mu_{\operatorname{bds}}(Y_2)\}\leq \mu_{\operatorname{bds}}(Y_1\times_k Y_2) \leq \mu_{\operatorname{bds}}(Y_1) \mu_{\operatorname{bds}}(Y_2)$.
    \end{enumerate}
\end{theorem}

Thus, we have some clarity:

\begin{corollary}
    [see \Cref{thm:base_change_bds_field} and \Cref{cor:fibered_product_bds}]
    \label{cor:intro_field_extensions_fiber_product}
    Let $k$ be perfect. Consider proper $k$-schemes $Y_1$ and $Y_2$ which admit resolutions of singularities. Then
    \begin{enumerate}
        \item \label{cor:intro_field_extensions_fiber_product1} $Y_1$ is a birational derived splinter if, and only if, $Y_1\times_k \operatorname{Spec}(\overline{k})$ is where $\overline{k}$ is an algebraic closure of $k$,
        \item \label{cor:intro_field_extensions_fiber_product2} $Y_1$ and $Y_2$ are birational derived splinters if, and only if, $Y_1\times_k Y_2$ is such.
    \end{enumerate}
\end{corollary}

To the best of our knowledge, these results are new in positive characteristic and reflects a few related results for rational singularities in characteristic zero. However, the techniques to prove these results differ from those in characteristic zero. Moreover, we do not know whether the converse direction of \Cref{cor:intro_field_extensions_fiber_product}\eqref{cor:intro_field_extensions_fiber_product1} (i.e.\ descending singularities)  has appeared before in the literature for rational singularities (partly due to the possibility of needing to deal with infinite direct sums). 

In summary, our work uses categorical techniques to study singularities. We believe that the results presented here, together with those in \cite{Lank/Venkatesh:2025}, represent strong strides towards understanding singularities through derived categorical methods.
Additionally, to make this work accessible to those from a different background, e.g.\ those not familiar with generation for triangulated categories or certain geometric techniques, we have spelled out the proofs with sufficient details.

\subsection*{Notation}
\label{sec:intro_notation}

Let $X$ be a Noetherian scheme. We work with the following triangulated categories: $D(X):=D(\operatorname{Mod}(X))$ the derived category of $\mathcal{O}_X$-modules; $D_{\operatorname{qc}}(X)$ the (strictly full) subcategory of $D(X)$ consisting of complexes with quasi-coherent cohomology; $D_{\operatorname{coh}}^b(X)$ the (strictly full) subcategory of $D(X)$ consisting of complexes having bounded and coherent cohomology; and $\operatorname{Perf}(X)$ the (strictly full) subcategory of $D_{\operatorname{qc}}(X)$ consisting of the perfect complexes on $X$.
At times when $X$ is affine, we abuse notation and write $D(R):=D_{\operatorname{qc}}(X)$ where $R:=H^0(X,\mathcal{O}_X)$ are the global sections; similar abuse holds for other categories. 
Furthermore, given a morphism $f\colon Y \to X$, we write $\mathcal{O}_X \xrightarrow{ntrl.} \mathbf{R}f_\ast \mathcal{O}_Y$ for the natural morphism, i.e.\ the unit of the derived pullback/pushforward adjunction.

\begin{ack}
    De Deyn was supported by ERC Consolidator Grant 101001227 (MMiMMa).
    Lank was supported under the ERC Advanced Grant 101095900-TriCatApp, partly while visiting University of Glasgow and the Basque Center for Applied Mathematics (BCAM), and thanks S\'{a}ndor Kov\'{a}cs and Karl Schwede for comments on an earlier version of this work. The authors thank the referee for helpful suggestions and a careful reading.
\end{ack}

\section{Preliminaries}
\label{sec:prelim}

\subsection{Generation}
\label{sec:prelim_generation}

We use notions of generation in triangulated categories as introduced in \cite{Bondal/VandenBergh:2003, Avramov/Buchweitz/Iyengar/Miller:2010}. The reader is encouraged to refer to loc.\ cit.\ for more details. Let $\mathcal{T}$ be a triangulated category. 
A functor $F\colon \mathcal{T}\to \mathcal{T}^\prime$ between triangulated categories is called \textbf{essentially surjective} (resp.\ \textbf{dense}) if every object $T^\prime \in \mathcal{T}$ is isomorphic to (resp.\ a retract of) $F(T)$ where $T\in \mathcal{T}$. A \textbf{strictly full} subcategory is a full subcategory closed under isomorphisms. We say that a subcategory $\mathcal{S}$ of $\mathcal{T}$ is \textbf{thick} if it is a triangulated subcategory of $\mathcal{T}$ closed under direct summands.
Generally, subcategories are not thick (or triangulated).
For any collection of objects $S$, let $\operatorname{add}(S)$ be the smallest strictly full subcategory of $\mathcal{T}$ containing $S$ that is closed under shifts, finite coproducts and direct summands. 
Inductively define
\begin{displaymath}
    \langle S\rangle_n :=
    \begin{cases}
        \operatorname{add}(0) & n=0, \\
        \operatorname{add}(S) & n=1, \\
        \operatorname{add}(\{ \operatorname{cone}\phi \mid \phi \in \operatorname{Hom}(\langle \mathcal{S} \rangle_{n-1}, \langle \mathcal{S} \rangle_1) \}) & n>1.
    \end{cases}
\end{displaymath}
An important feature of this construction is that $\langle S \rangle = \cup_{n\geq 0} \langle  S \rangle_n$, where $\langle S \rangle$ is the smallest thick subcategory of $\mathcal{T}$ that contains $S$.  We say an object $E\in \mathcal{T}$ is \textbf{finitely built by $\mathcal{S}$} if $E\in \langle \mathcal{S} \rangle$. 
The smallest $n$ for which $E\in \langle \mathcal{S} \rangle_n$ is called the \textbf{level of $E$ with respect to $\mathcal{S}$}. This value is denoted by $\operatorname{level}^{\mathcal{S}}(E)$ (otherwise, set to $\infty$).

\subsection{Types of morphisms}
\label{sec:prelim_singularities_morphisms}

Let $f\colon Y \to X$ be a morphism of integral Noetherian schemes. We say $f$ is a \textbf{modification} if it is proper and birational, see e.g.\ \cite[\href{https://stacks.math.columbia.edu/tag/01RN}{Tag 01RN}]{StacksProject} for some characterizations. Also, $f$ is an \textbf{alteration} if it is proper, dominant and generically finite, see e.g.\ \cite[\href{https://stacks.math.columbia.edu/tag/02NV}{Tag 02NV}]{StacksProject} for some characterizations.
Observe that each morphism above is surjective because each morphism is proper and dominant. Indeed, it follows from that fact each map on underlying topological spaces is closed and its image contains the generic point.
By \cite[\href{https://stacks.math.columbia.edu/tag/080B}{Tag 080B}]{StacksProject}, a sequence of blowups of Noetherian schemes is a (single) blowup. As a convention, we say a modification is projective if the morphism is such, and analogously for other adjectives. A useful source of modifications are blowups of integral Noetherian schemes along nonzero ideal sheaves (see e.g.\ \cite[\href{https://stacks.math.columbia.edu/tag/02ND}{Tag 02ND} \& \href{https://stacks.math.columbia.edu/tag/02OS}{Tag 02OS}]{StacksProject}). 

Now, for the general setting of Noetherian schemes which need not be integral, we follow \cite[\href{https://stacks.math.columbia.edu/tag/01RN}{Tag 01RN}]{StacksProject} for the definition of birationality. A fact we use liberally is that blowups are compatible with flat base change (see e.g.\ \cite[\href{https://stacks.math.columbia.edu/tag/0805}{Tag 0805}]{StacksProject})and that birationality of a morphism base changes along flat morphisms (see e.g.\ \cite[\S 1, Proposition 3.9.9]{Grothendieck/Dieudonne:1971}). Moreover, we say a \textbf{resolution of singularities} is a proper birational morphism from a regular scheme. 

\subsection{Birational derived splinters}
\label{sec:birational_derived_splinters}

We discuss a notion of singularities first introduced by Kov\'{a}cs. Recall a Noetherian scheme $X$ is called a \textbf{birational derived splinter} if $\mathcal{O}_X \xrightarrow{ntrl.} \mathbf{R}f_\ast \mathcal{O}_Y$ splits for every proper birational morphism $f\colon Y \to X$.
For example, any quasi-compact regular scheme is a birational derived splinter (use \Cref{lem:LV25}). We record a few useful observations.

\begin{lemma}
    \label{lem:splitting_lemma_compositions}
    Let $f\colon Y \to X$ and $g\colon Z \to Y$ be morphisms of Noetherian schemes. If $\mathcal{O}_X \xrightarrow{ntrl.} \mathbf{R} (f \circ g)_\ast \mathcal{O}_Z$ splits, then $\mathcal{O}_X \xrightarrow{ntrl.} \mathbf{R} f_\ast \mathcal{O}_Y$ splits.
\end{lemma}

\begin{proof}
    This follows from $\mathcal{O}_X \xrightarrow{ntrl.} \mathbf{R} (f \circ g)_\ast \mathcal{O}_Z$ factoring through $\mathcal{O}_X \xrightarrow{ntrl.} \mathbf{R} f_\ast \mathcal{O}_Y$.
\end{proof}

\begin{proposition}
    \label{prop:birational_derived_splinter_via_blow_ups}
    Let $X$ be an integral Noetherian scheme. Then $X$ is a birational derived splinter if, and only if, $\mathcal{O}_X \xrightarrow{ntrl.} \mathbf{R}f_\ast \mathcal{O}_{X^\prime}$ splits for every blowup of $X$ along a nonzero ideal sheaf.
\end{proposition}

\begin{proof}
    If $X$ is a birational derived splinter, then the desired spitting condition is obvious. So, we check the converse. Using \cite[Lemma 2.2]{Lutkebohmert:1993} and \Cref{lem:splitting_lemma_compositions}, one can reduce to the case of blowups. Indeed, loc.\ cit.\ allows us to dominate any proper birational morphism $Y\to X$ by a blowup $X^\prime \to X$ which factors through $Y\to X$.
\end{proof}

\begin{proposition}
    \label{prop:splitting_birational_derived_splinter}
    Let $f\colon Y \to S$ be a proper surjective morphism of Noetherian schemes such that $\mathcal{O}_S \xrightarrow{ntrl.} \mathbf{R}f_\ast \mathcal{O}_Y$ splits. If $Y$ is a birational derived splinter, then so is $S$.
\end{proposition}

\begin{proof}
    This follows by arguing with \cite[Lemma 2.2]{Lutkebohmert:1993} but we add details for interest of techniques. Let $g\colon X \to S$ be a proper birational morphism. Denote by $g^\prime\colon Y\times_S X \to Y$ for the projection morphism. Using \cite[Lemma 2.2]{Lutkebohmert:1993}, we can find a proper birational morphism $h\colon Z \to Y$ which factors through $g^\prime$.
    Indeed, from base change, we see that $g^\prime$ is a proper morphism which is an isomorphism over an open subset of $Y$. 
    If $Y$ is a birational derived splinter, then $\mathcal{O}_Y \xrightarrow{ntrl.} \mathbf{R}h_\ast \mathcal{O}_Z$ splits. Hence, by \Cref{lem:splitting_lemma_compositions}, $\mathcal{O}_Y \xrightarrow{ntrl.} \mathbf{R}g^\prime_\ast \mathcal{O}_{Y\times_S X}$ splits. So, the claim follows from \Cref{lem:splitting_lemma_compositions}.
\end{proof}

The following is very useful for our work.

\begin{lemma}
    [{\cite[Lemma 3.16]{Lank/Venkatesh:2025}}]
    \label{lem:LV25}
    Let $f\colon Y \to X$ be a proper birational morphism to a Noetherian scheme. If $\mathbf{R}f_\ast \mathcal{O}_Y \in \operatorname{Perf}(X)$ (e.g. if $X$ is regular), then the natural morphism $\mathcal{O}_X \to \mathbf{R}f_\ast \mathcal{O}_Y$ splits.
\end{lemma}

\begin{proof}
    This lemma was communicated by Bhatt to Lank and Venkatesh in \cite{Lank/Venkatesh:2025}. 
    However, for completeness sake, we add a proof in the special case where $X$ is reduced; this suffices for our purposes below.
    The proof in full generality requires the use of `derived algebraic geometry' and is out of scope for this note.
    
    Let $U$ denote the dense open locus of $X$ over which $f$ is an isomorphism. 
    Observe that, as we assume $X$ is reduced, this open is scheme-theoretically dense by \cite[\href{https://stacks.math.columbia.edu/tag/056D}{Tag 056D}]{StacksProject}.
    Crucially, it follows that the restriction $\rho\colon \Gamma(X,\mathcal{O}_X)\to \Gamma(U,\mathcal{O}_X)$ is injective.

    Now, as $\mathbf{R}f_\ast \mathcal{O}_Y$ is perfect we have a trace morphism (coming from `evaluation') 
    \begin{displaymath}
        \operatorname{\mathbf{R}\mathcal{E}\! \mathit{nd}}(\mathbf{R}f_\ast \mathcal{O}_Y)\cong \operatorname{\mathbf{R}\mathcal{H}\! \mathit{om}}(\mathbf{R}f_\ast \mathcal{O}_Y, \mathcal{O}_X)\otimes_{\mathcal{O}_X} \mathbf{R}f_\ast \mathcal{O}_Y \to \mathcal{O}_X.
    \end{displaymath}
    Pre-composing with the morphism $\mathcal{O}_X\to \operatorname{\mathbf{R}\mathcal{E}\! \mathit{nd}}(\mathbf{R}f_\ast \mathcal{O}_Y)$, obtained from the $\mathcal{O}_X$-linear structure on the Hom-sets, we obtain an endomorphism of $\mathcal{O}_X$ which we can view as a global section $s\in \Gamma(X,\mathcal{O}_X)$.
    
    Note that the morphism $\mathcal{O}_X\to \operatorname{\mathbf{R}\mathcal{E}\! \mathit{nd}}(\mathbf{R}f_\ast \mathcal{O}_Y)$ factors through the natural morphism $\mathcal{O}_X \to \mathbf{R}f_\ast \mathcal{O}_Y$; the morphism $\mathbf{R}f_\ast \mathcal{O}_Y\to \operatorname{\mathbf{R}\mathcal{E}\! \mathit{nd}}(\mathbf{R}f_\ast \mathcal{O}_Y)$ coming from the algebra structure on the former.
    Hence is suffices to show $s=1$, which follows if we can show $\rho(s)=1$.
    By choice of $U$
    \begin{displaymath}
        \operatorname{\mathbf{R}\mathcal{E}\! \mathit{nd}}(\mathbf{R}f_\ast \mathcal{O}_Y)|_U\cong \mathcal{O}_U.
    \end{displaymath}
    As this isomorphism is compatible with the trace morphisms and the $\mathcal{O}_X$-linear structure, it follows that $\rho(s)=1$, finishing the proof.
    \end{proof}

\begin{theorem}
    \label{thm:Murayama_criterion}
    Let $X$ be a Noetherian scheme that admits a resolution of singularities $f\colon \widetilde{X} \to X$. Then the following are equivalent:
    \begin{enumerate}
        \item \label{thm:Murayama_criterion1} $\mathcal{O}_X \xrightarrow{ntrl.} \mathbf{R} f_\ast \mathcal{O}_{\widetilde{X}}$ splits
        \item \label{thm:Murayama_criterion2} $X$ is a birational derived splinter.
    \end{enumerate}
\end{theorem}

\begin{proof}
    $\eqref{thm:Murayama_criterion2}\implies \eqref{thm:Murayama_criterion1}$ is obvious and the converse follows via \Cref{prop:splitting_birational_derived_splinter} and the fact that any quasi-compact regular scheme is a birational derived splinter by \Cref{lem:LV25}.
\end{proof}

\begin{corollary}
    [cf.\ {\cite[Proposition 3.5]{Lyu:2022}}]
    \label{cor:birational_notions_for_quasi_excellent_char_zero}
    Let $X$ be a quasi-compact quasi-excellent integral $\mathbb{Q}$-scheme. Then the following are equivalent:
    \begin{enumerate}
        \item \label{cor:birational_notions_for_quasi_excellent_char_zero1} $X$ is a derived splinter
        \item \label{cor:birational_notions_for_quasi_excellent_char_zero2} $X$ is a birational derived splinter.
    \end{enumerate}
\end{corollary}

\begin{proof}
    It is obvious that $\eqref{cor:birational_notions_for_quasi_excellent_char_zero1} \implies \eqref{cor:birational_notions_for_quasi_excellent_char_zero2}$. By \cite[Theorem 1.1]{Temkin:2008}, there exists a modification of $X$ from a regular integral scheme. So, \Cref{thm:Murayama_criterion} coupled with \cite[Theorem 9.5]{Murayama:2025} implies the converse where loc.\ cit.\ states that derived splinters in characteristic zero are exactly rational singularities.
\end{proof}

\begin{example}
    [Differences in prime characteristic]
    \label{ex:bhatt_elliptic_curve}
    Let $E$ be an elliptic curve over a field of prime characteristic $p$. The multiplication by $p$ morphism $[p]\colon E \to E$ is finite and surjective. From \cite[Example 2.11]{Bhatt:2012}, we know that $\mathcal{O}_E \xrightarrow{ntrl.} [p]_\ast \mathcal{O}_E$ does not split because the induced map on $H^1$ is zero, so $E$ cannot be a derived splinter. Yet, as it is regular, \Cref{lem:LV25} shows it is a birational derived splinter. 
\end{example}

We record the following lemmas for interest sake.

\begin{lemma}
    \label{lem:isomorphism_from_section}
    Let $X$ be an integral Noetherian scheme. 
    Suppose $\phi\colon A \to B$ is a morphism of torsion free coherent $\mathcal{O}_X$-modules of same rank. 
    If $\phi$ admits a right (resp.\ left) inverse, then $\phi$ is an isomorphism.
\end{lemma}

\begin{proof}
    This is well-known, but we spell it out for sake of convenience. Let $U=\operatorname{Spec}(R)$ be any affine open in $X$. First, suppose $\phi$ admits a right inverse; in particular, $\phi$ is an epimorphism. As $X$ is integral, the local ring at the generic point $\xi$ is a field. Consequently, $\phi_\xi$ is an isomorphism as it is an epimorphism between vector spaces of the same dimension. It follows that $\operatorname{ker}(\phi)(U)$ is a torsion $R$-module, which is necessarily zero because $A(U)$ is torsion free by \cite[\href{https://stacks.math.columbia.edu/tag/0AXS}{Tag 0AXS}]{StacksProject}. Consequently, for every affine open $U$ in $X$, $\phi|_U$ is an isomorphism as we already know it is an epimorphism. Hence, $\phi$ is an isomorphism.
    
    In the case $\phi$ admits a left inverse, $\phi$ is a monomorphism. So, from similar reasoning, it follows that $\operatorname{coker}(\phi)(U)$ is a torsion $R$-module. As $B(U)$ is torsion free and $\operatorname{coker}(\phi)(U)$ is a submodule of this by the splitting, it follows that $\operatorname{coker}(\phi)(U)=0$. Thus, $\phi$ is an isomorphism, which completes the proof.
\end{proof}

\begin{definition}
    [cf.\ {\cite[Cor.\ on pg.\ 107]{Lipman/Teissier:1981}, see also \cite[proof of Proposition 12.5]{Krah/Vial:2023}}]
    A quasi-compact normal excellent Cohen-Macaulay scheme $X$ admitting a dualizing complex is called \textbf{pseudo-rational} if $f_\ast \omega_Y \xrightarrow{ntrl.} \omega_X$ is an isomorphism for every proper birational morphism $f\colon Y \to X$ from a normal scheme.
\end{definition}

\begin{lemma}
    \label{lem:pseudorational_bds_same}
    Let $X$ be a quasi-compact irreducible excellent normal Cohen-Macaulay scheme admitting a dualizing sheaf $\omega_X$. Then $X$ is a birational derived splinter if, and only if, it is pseudorational. 
\end{lemma}

\begin{proof}
    By \cite[Lemma 3.8]{Ma/McDonald/RG/Schwede:2025}, pseudorationality implies being a birational derived splinter. So, we check the converse. Let $f\colon Y\to X$ be a projective modification from a normal scheme. The hypothesis tells us $\mathcal{O}_X \xrightarrow{ntrl.} \mathbf{R}f_\ast \mathcal{O}_Y$ admits a left inverse. Applying the functor $\operatorname{\mathbf{R}\mathcal{H}\! \mathit{om}} (-,\omega_X^\bullet)$, using Grothendieck duality, and taking cohomology in the appropriate degree, we see that $f_\ast \omega_Y \xrightarrow{ntrl.} \omega_X$ admits a right inverse. Consequently, \Cref{lem:isomorphism_from_section} implies $f_\ast \omega_Y \xrightarrow{ntrl.} \omega_X$ is an isomorphism. 
    Indeed, $f_\ast \omega_Y \xrightarrow{ntrl.} \omega_X$ is an isomorphism at all $p\in X$ such that $\dim \mathcal{O}_{X,p}=1$ and $f_\ast \omega_Y$ satisfies $(S_2)$ (see e.g.\ \cite[\href{https://stacks.math.columbia.edu/tag/0BFP}{Tag 0BFP}]{StacksProject}). Thus, $X$ is pseudorational.
\end{proof} 

\section{Measurements}
\label{sec:measurements}

\subsection{First steps}
\label{sec:measurements_first steps}

The following is the main notion of this work.

\begin{definition}
    \label{def:bds_measurements}
    Let $X$ be a Noetherian scheme. 
    Define
    \begin{displaymath}
        \mu_{\operatorname{bds}}(X) := \sup\left\{ \operatorname{level}^{\mathbf{R}f_\ast D^b_{\operatorname{coh}}(Y) } (\mathcal{O}_X) \mid f\colon Y \to X \textrm{ proper and birational} \right\}.
    \end{displaymath}
    If $X=\operatorname{Spec}(R)$ we also write $\mu_{\operatorname{bds}}(R)$ instead of $\mu_{\operatorname{bds}}(X)$.
\end{definition}

By \Cref{lem:thick_one_iff_splitting_naturally}, $X$ is a birational derived splinter precisely when $\mu_{\operatorname{bds}}(X) =1$. The following shows that in fact what one is measuring is not just the structure sheaf, but all of the perfect complexes.

\begin{lemma}
    \label{lem:characterize_by_perfects}
    Let $f\colon Y \to X$ be a proper morphism of Noetherian schemes. 
    For any $n\geq 0$, one has $\mathcal{O}_X \in\langle \mathbf{R}f_\ast D^b_{\operatorname{coh}}(Y) \rangle_n$ if and only if $\operatorname{Perf}(X) \subseteq \langle \mathbf{R}f_\ast D^b_{\operatorname{coh}}(Y) \rangle_n$.
\end{lemma}

\begin{proof}
    Assume that $\mathcal{O}_X \in\langle \mathbf{R}f_\ast D^b_{\operatorname{coh}}(Y) \rangle_n$. Choose $P\in \operatorname{Perf}(X)$. 
    We have
    \begin{displaymath}
        P \otimes^{\mathbf{L}} \langle \mathbf{R}f_\ast D^b_{\operatorname{coh}}(Y) \rangle_n \subseteq \langle \mathbf{R}f_\ast (\mathbf{L}f^\ast P \otimes^{\mathbf{L}} D^b_{\operatorname{coh}}(Y)) \rangle_n \subseteq \langle \mathbf{R}f_\ast D^b_{\operatorname{coh}}(Y) \rangle_n,   
    \end{displaymath}
    see e.g.\ \cite[Lemma 3.5]{Lank:2024} for the first inclusion.
    Thus, tensoring the assumption by $P$ yields $P\in \langle \mathbf{R}f_\ast D^b_{\operatorname{coh}}(Y) \rangle_n$ showing the desired claim.
    The converse direction is clear, completing the proof.
\end{proof}

\begin{proposition}
    \label{prop:ds_by_alteration}
    Let $X$ be a Noetherian scheme. 
    Then $\mu_{\operatorname{bds}}(X)$ is the smallest $N\in \mathbb{Z}^+\cup\{\infty\}$ such that $\mathcal{O}_X \in \langle \mathbf{R}f_\ast D^b_{\operatorname{coh}}(X^\prime)\rangle_N$ for all blowups $f\colon X^\prime \to X$ along ideal sheaves so that birationality occurs (see e.g.\ \cite[\href{https://stacks.math.columbia.edu/tag/0BFM}{Tag 0BFM}]{StacksProject}).
\end{proposition}

\begin{proof}
    For brevity, let's call an ideal sheaf `suitable' if birationality occurs for the associated blowup. Here, $N$ is the minimal element in $\mathbb{Z}^+\cup\{\infty\}$ such that $\mathcal{O}_X \in \langle \mathbf{R}f_\ast D^b_{\operatorname{coh}}(X^\prime)\rangle_N$ for all blowups along suitable ideal sheaves. Clearly, $N\leq \mu_{\operatorname{bds}}(X)$ because any blowup along a nonzero suitable ideal sheaf is proper and birational.
    Moreover, if $N=\infty$ there is nothing to show, so lets assume $N<\infty$. Let $g\colon Y \to X$ be an arbitrary proper birational morphism. Using \cite[Lemma 2.2]{Lutkebohmert:1993}, we can find a proper birational morphism $h\colon Z \to Y$ such that $g\circ h$ is itself a blowup. 
    Hence, $\mathcal{O}_X \in \langle \mathbf{R}(g\circ h)_\ast D^b_{\operatorname{coh}}(Z)\rangle_N$, which implies $\mathcal{O}_X\in \langle \mathbf{R}g_\ast D^b_{\operatorname{coh}}(Y) \rangle_N$. Thus,  $\mu_{\operatorname{bds}}(X)\leq N$ as $g$ was arbitrary. 
\end{proof}

\begin{theorem}
    \label{thm:r_value_for_modification_finite}
    Let $X$ be a Noetherian scheme which admits a resolution of singularities $f\colon \widetilde{X}\to X$. Then $\mu_{\operatorname{bds}}(X)=\operatorname{level}^{\mathbf{R}f_\ast D^b_{\operatorname{coh}}\left(\widetilde{X}\right)}(\mathcal{O}_X) < \infty$.
\end{theorem}

\begin{proof}
    To start, \cite[Lemma 3.9]{Dey/Lank:2024} tells us that $\operatorname{level}^{\mathbf{R}f_\ast D^b_{\operatorname{coh}}\left(\widetilde{X}\right)} (\mathcal{O}_X) < \infty$. Let $g\colon Y \to X$ be an arbitrary proper birational morphism. By definition, we know that $\operatorname{level}^{\mathbf{R}f_\ast D^b_{\operatorname{coh}}(\widetilde{X})} (\mathcal{O}_X) \leq \mu_{\operatorname{bds}}(X)$. So, as $g$ is an arbitrary proper birational morphism, we only need to show $\operatorname{level}^{\mathbf{R}g_\ast D^b_{\operatorname{coh}}(Y)} (\mathcal{O}_X) \leq \operatorname{level}^{\mathbf{R}f_\ast D^b_{\operatorname{coh}}(\widetilde{X})} (\mathcal{O}_X)$. 
    Using \cite[Lemma 2.2]{Lutkebohmert:1993}, there is a commutative diagram
    \begin{displaymath}
        \begin{tikzcd}
            {Z} & {\widetilde{X}} \\
            Y & X
            \arrow["{h^\prime}", from=1-1, to=1-2]
            \arrow["h"', from=1-1, to=2-1]
            \arrow["f", from=1-2, to=2-2]
            \arrow["g"', from=2-1, to=2-2]
        \end{tikzcd}
    \end{displaymath}
    where $h$ and $h^\prime$ are proper birational morphisms. By \Cref{lem:LV25}, we have that $\mathcal{O}_{\widetilde{X}}\xrightarrow{ntrl.}\mathbf{R}h^\prime_\ast \mathcal{O}_{Z}$ splits because $\mathbf{R}h^\prime_\ast \mathcal{O}_{Z}$ is perfect as the target is regular. From \Cref{lem:characterize_by_perfects}, it follows that $D^b_{\operatorname{coh}}(\widetilde{X}) = \langle \mathbf{R}h^\prime_\ast D^b_{\operatorname{coh}}(Z) \rangle_1$ as $\widetilde{X}$ is regular. Moreover, \cite[Lemma 5.5]{DeDeyn/Lank/ManaliRahul:2025} gives us an $E\in D^b_{\operatorname{coh}}(\widetilde{X})$ such that $N:=\operatorname{level}^{\mathbf{R}f_\ast E} (\mathcal{O}_X)=\operatorname{level}^{\mathbf{R}f_\ast D^b_{\operatorname{coh}}(\widetilde{X})} (\mathcal{O}_X)$. 
    Tying things together, we can choose $E^\prime\in D^b_{\operatorname{coh}}(Z)$ such that $E\in \langle \mathbf{R}h^\prime_\ast E^\prime \rangle_1$, and so $\mathcal{O}_X \in \langle \mathbf{R}(f\circ h^\prime)_\ast  E^\prime \rangle_N$ (see e.g.\ \cite[Lemma 2.4]{Avramov/Buchweitz/Iyengar/Miller:2010}). Then the string of inequalities
    \begin{displaymath}
        \begin{aligned}
            \operatorname{level}^{\mathbf{R}g_\ast D^b_{\operatorname{coh}}(Y)} (\mathcal{O}_X)
            &\leq \operatorname{level}^{\mathbf{R}g_\ast D^b_{\operatorname{coh}}(Y)} (\mathbf{R}(g\circ h)_\ast E^\prime) \operatorname{level}^{\mathbf{R}(g\circ h)_\ast E^\prime} (\mathcal{O}_X)
            \\&\leq \operatorname{level}^{\mathbf{R}(g\circ h)_\ast E^\prime} (\mathcal{O}_X)
            \\&=\operatorname{level}^{\mathbf{R}(f\circ h^\prime)_\ast E^\prime} (\mathcal{O}_X) 
            \\&\leq N=\operatorname{level}^{\mathbf{R}f_\ast D^b_{\operatorname{coh}}(\widetilde{X})} (\mathcal{O}_X)
        \end{aligned}
    \end{displaymath}
    shows the desired inequality, and so, $\mu_{\operatorname{bds}}(X)\leq N$.
\end{proof}

\subsection{Behavior under pullback}
\label{sec:measurements_pullback}

In this subsection we study the behaviors of $\mu_{\operatorname{bds}}$ under pullbacks.

\begin{corollary} 
    \label{cor:completions_bds}
    Let $(R,\mathfrak{m})$ be a local ring which is quasi-excellent and normal. Assume there is a resolution of singularities $f\colon \widetilde{X} \to \operatorname{Spec}(R)$. Then $\mu_{\operatorname{bds}}(R)\geq \mu_{\operatorname{bds}}(\widehat{R})$ where $\widehat{R}$ is the $\mathfrak{m}$-adic completion of $R$.
\end{corollary}

\begin{proof}
    From \cite[\href{https://stacks.math.columbia.edu/tag/0C23}{Tag 0C23}]{StacksProject}, we know that $\widehat{R}$ is a normal integral local ring. Denote by $\pi\colon \operatorname{Spec}(\widehat{R}) \to \operatorname{Spec}(R)$ the morphism induced by $R\to \widehat{R}$. Consider the fibered square
    \begin{displaymath}
        \begin{tikzcd}
            {\widetilde{X}\times_{\operatorname{Spec}(R)} \operatorname{Spec}(\widehat{R})} & {\operatorname{Spec}(\widehat{R})} \\
            {\widetilde{X}} & {\operatorname{Spec}(R)}\rlap{ .}
            \arrow["{f^\prime}", from=1-1, to=1-2]
            \arrow["{\pi^\prime}"', from=1-1, to=2-1]
            \arrow["\pi", from=1-2, to=2-2]
            \arrow["f"', from=2-1, to=2-2]
        \end{tikzcd}
    \end{displaymath}
    Then $f^\prime$ is a resolution of singularities (see e.g.\ \cite[\href{https://stacks.math.columbia.edu/tag/0BG6}{Tag 0BG6}]{StacksProject}). Suppose $\mathcal{O}_{\operatorname{Spec}(R)}\in \langle \mathbf{R}f_\ast D^b_{\operatorname{coh}}(\widetilde{X}) \rangle_n$. We see from flat base change that $\mathcal{O}_{\operatorname{Spec}(\widehat{R})} \in \langle \mathbf{R}f^\prime_\ast (\pi^\prime)^\ast D^b_{\operatorname{coh}}(\widetilde{X}) \rangle_n$. Hence, from the inclusion
    \begin{displaymath}
        \langle \mathbf{R}f^\prime_\ast (\pi^\prime)^\ast D^b_{\operatorname{coh}}(\widetilde{X}) \rangle_n \subseteq \langle \mathbf{R}f^\prime_\ast D^b_{\operatorname{coh}}(X\times_{\operatorname{Spec}(R)}\operatorname{Spec}(\widehat{R}))\rangle_n,
    \end{displaymath}
    the desired inequality follows by \Cref{thm:r_value_for_modification_finite}.
\end{proof}


\begin{corollary}
    \label{cor:localization_bds}
    Let $R$ be a reduced Noetherian ring. Assume there is a resolution of singularities $f\colon \widetilde{X} \to \operatorname{Spec}(R)$. For any multiplicatively closed subset $S\subseteq R$, one has $\mu_{\operatorname{bds}}(S^{-1} R)\leq \mu_{\operatorname{bds}}(R)$ where $S^{-1}R$ is the localization of $R$ at $S$.
\end{corollary}

\begin{proof}
    This is argued essentially as in the proof of \Cref{cor:completions_bds} but we spell out a few details. First, the natural morphism $\operatorname{Spec}(S^{-1}R) \to \operatorname{Spec}(R)$ is a flat morphism of reduced Noetherian schemes. 
    So, the base change of $f$ along $\operatorname{Spec}(S^{-1}R) \to \operatorname{Spec}(R)$ gives a proper birational morphism to $\operatorname{Spec}(S^{-1}R)$.
    Moreover, the source of the morphism obtained from base change is regular. Hence, we can proceed by a similar argument as in \Cref{cor:completions_bds}.
\end{proof}


\begin{proposition}
    \label{prop:subadditive}
    Let $f\colon Y\to X$ be a smooth morphism of Noetherian schemes. Assume $X$ admits a resolution of singularities $g\colon \widetilde{X}\to X$. Then $\mu_{\operatorname{bds}}(X) \geq \mu_{\operatorname{bds}}(Y)$.
\end{proposition}

\begin{proof}
    Consider the fibered square
    \begin{displaymath}
        \begin{tikzcd}
            {\widetilde{X}\times_X Y} & Y \\
            {\widetilde{X}} & X\rlap{ .}
            \arrow["{g^\prime}", from=1-1, to=1-2]
            \arrow["{f^\prime}"', from=1-1, to=2-1]
            \arrow["f", from=1-2, to=2-2]
            \arrow["g"', from=2-1, to=2-2]
        \end{tikzcd}
    \end{displaymath}
    As $f$ is flat, it follows $g^\prime$ is proper and birational. Moreover, $\widetilde{X}\times_X Y$ is smooth over $\widetilde{X}$, and so by \cite[Proposition 17.5.8]{Grothendieck:1965} $\widetilde{X}\times_X Y$ is regular. Hence, $g^\prime$ can be used to compute $\mu_{\operatorname{bds}}(Y)$ by \Cref{thm:r_value_for_modification_finite}. 
    Choose $n\geq 1$ such that $\mathcal{O}_X \in \langle \mathbf{R}f_\ast D^b_{\operatorname{coh}}(\widetilde{X}) \rangle_n$. By flat base change, we see that 
    \begin{displaymath}
        \begin{aligned}
            \mathcal{O}_Y
            &\in \langle \mathbf{L}f^\ast \mathbf{R}g_\ast D^b_{\operatorname{coh}}(\widetilde{X})\rangle_n
            \\&\subseteq \langle \mathbf{R} g^\prime_\ast \mathbf{L} (f^\prime)^\ast D^b_{\operatorname{coh}}(\widetilde{X})\rangle_n
            \\&\subseteq \langle \mathbf{R} g^\prime_\ast D^b_{\operatorname{coh}}(\widetilde{X}\times_X Y)\rangle_n.
        \end{aligned}
    \end{displaymath}
    Then the desired inequality follows immediately.
\end{proof}

\begin{corollary}
    \label{cor:bds_bundles}
    Let $X$ be a Noetherian scheme which admits a resolution of singularities $f\colon \widetilde{X}\to X$. Suppose $\mathcal{E}$ is a vector bundle of rank $r+1$ on $X$ ($r\geq 1$). Then $\mu_{\operatorname{bds}}(X) = \mu_{\operatorname{bds}}(\mathbb{P}_X (\mathcal{E}))$.
\end{corollary}

\begin{proof}
    One inequality comes from \Cref{prop:subadditive} as $\pi\colon\mathbb{P}_X (\mathcal{E})\to X$ is smooth. For the other inequality, consider the fibered square
    \begin{displaymath}
        \begin{tikzcd}
            {\mathbb{P}_X (f^\ast\mathcal{E})} & {\widetilde{X}} \\
            {\mathbb{P}_X (\mathcal{E})} & {X}\rlap{ .}
            \arrow["{\pi^\prime}"', from=1-1, to=1-2]
            \arrow["{f^\prime}", from=1-1, to=2-1]
            \arrow["f", from=1-2, to=2-2]
            \arrow["\pi", from=2-1, to=2-2]
        \end{tikzcd}
    \end{displaymath}
    Base change tells us that $\pi^\prime$ is a smooth morphism. Hence, $f^\prime$ is proper and birational. As $\mathcal{O}_X \to \mathbf{R} \pi_\ast \mathcal{O}_{\mathbb{P}_X (\mathcal{E})}$ is an isomorphism (see e.g.\ \cite[\S 3, Exercise 8.4]{Hartshorne:1983}), flat base change tells us $\mathcal{O}_{\widetilde{X}} \xrightarrow{ntrl.} \mathbf{R} \pi^\prime_\ast \mathcal{O}_{\mathbb{P}_X (f^\ast\mathcal{E}) }$ is as well.
    Then $\mathbf{R}\pi_\ast^\prime \colon D^b_{\operatorname{coh}}(\mathbb{P}_X (f^\ast\mathcal{E}) ) \to D^b_{\operatorname{coh}}(\widetilde{X})$ is essentially surjective. Hence, from commutativity of the diagram above, we have that
    \begin{displaymath}
        \operatorname{level}^{\mathbf{R}f_\ast D^b_{\operatorname{coh}}(\widetilde{X})} (\mathcal{O}_X) = \operatorname{level}^{\mathbf{R}(f\circ \pi^\prime)_\ast D^b_{\operatorname{coh}}(\mathbb{P}_X (f^\ast\mathcal{E}) )} (\mathcal{O}_X).
    \end{displaymath}
    But we know that 
    \begin{displaymath}
       \mathcal{O}_X \cong \mathbf{R} \pi_\ast \mathcal{O}_{\mathbb{P}_X (\mathcal{E})}\in \langle \mathbf{R}(\pi\circ f^\prime)_\ast D^b_{\operatorname{coh}}(\mathbb{P}_X (f^\ast\mathcal{E}) ) \rangle_{\mu_{\operatorname{bds}}(\mathbb{P}_X (\mathcal{E}))}.
    \end{displaymath}
    This tells us $\mu_{\operatorname{bds}}(X)\leq \mu_{\operatorname{bds}}(\mathbb{P}_X (\mathcal{E}))$ via \Cref{thm:r_value_for_modification_finite}. So, we are done.
\end{proof}

\begin{remark}\label{rem:generalizations}
    \Cref{cor:bds_bundles} holds more generally for any smooth morphism whose (derived) unit is an isomorphism on structure sheaves. Indeed, the same argument as above will work out.
\end{remark}

\subsection{Local behavior}
\label{sec:measurements_local_behavior}

Next, we focus on the case of $\mu_{\operatorname{bds}}$ for affine schemes and show that this is compatible with looking at stalks. It is not known as to whether being a birational derived splinter is a stalk local property for general schemes. So, any global versions of \Cref{prop:measurement_for_affine_local_to_global} we believe to be more difficult.

\begin{lemma}
    \label{lem:base_change_essentially_surjective}
    Let $f\colon Y \to X$ be a proper surjective morphism to a separated Noetherian scheme and $p\in X$. Consider the fibered square
    \begin{displaymath}
    \begin{tikzcd}
    	{Y\times_X \operatorname{Spec}(\mathcal{O}_{X,p})} & {\operatorname{Spec}(\mathcal{O}_{X,p})} \\
    	Y & {X}\rlap{ .}
    	\arrow["{f^\prime}", from=1-1, to=1-2]
    	\arrow["{s^\prime}", from=1-1, to=2-1]
    	\arrow["s"', from=1-2, to=2-2]
    	\arrow["f", from=2-1, to=2-2]
    \end{tikzcd}
    \end{displaymath}
    Then $\mathbf{L}(s^\prime)^\ast \colon D^b_{\operatorname{coh}}(Y) \to D^b_{\operatorname{coh}}(Y\times_X \operatorname{Spec}(\mathcal{O}_{X,p}))$ is essentially surjective.
\end{lemma}

\begin{proof}
    First, note $s$ is an affine morphism (see e.g.\ \cite[\href{https://stacks.math.columbia.edu/tag/01SG}{Tag 01SG}]{StacksProject}). Moreover, we know that $s$ is flat. Thus, as $s$ is flat and affine, so is $s^\prime$ by base change. Consequently, $s^\prime_\ast$ and $(s^\prime)^\ast$ are exact, which ensures $\mathbf{R}s^\prime_\ast=s^\prime_\ast$ and $\mathbf{L} (s^\prime)^\ast=(s^\prime)^\ast$.
    In order to show essential surjectivity, we first show the counit  
    \begin{equation}\label{eq:counit}
        (s^\prime)^\ast s^\prime_\ast E \to E
    \end{equation} 
    is an isomorphism for all $E\in D^b_{\operatorname{coh}}(Y\times_X \operatorname{Spec}(\mathcal{O}_{X,p}))$.
    To this end, note that this can be checked on cohomology; so we may, by exactness of $(s^\prime)^\ast$ and $s^\prime_\ast$, assume $E$ is a coherent sheaf. 
    Furthermore, showing it is an isomorphism is local on $Y$; so we may reduce to the case both $X$ and $Y$ are affine, i.e.\ pick an affine cover $U_i$ of $X$ and take an affine subcovering $V_{ij}$ of $f^{-1}(U_i)$)
    However, in this case $s^\prime$ is induced by a localization and it follows that \Cref{eq:counit} is an isomorphism.

    To finish, we bootstrap the `Observation' and `Essentially surjective' parts of \cite[Theorem 4.4]{Elagin/Lunts/Schnurer:2020}.
    By the argument of loc.\ cit.\ we can find a subcomplex $K\subseteq s^\prime_\ast E$ which is bounded and coherent such that $(s^\prime)^\ast K = (s^\prime)^\ast s^\prime_\ast E\cong E$.
\end{proof}

\begin{lemma}
    [cf.\ {\cite[Corollary 3.4]{Letz:2021}}]
    \label{lem:generalized_Letz_level}
    Let $X$ be a Noetherian affine scheme. Suppose $\mathcal{S}\subseteq D^b_{\operatorname{coh}}(X)$ and $E\in D^b_{\operatorname{coh}}(X)$. If $E\in \langle \mathcal{S} \rangle$, then $\operatorname{level}^{\mathcal{S}} (E) = {\sup}_{p \in X}\{\operatorname{level}^{\mathcal{S}_p} (E_p) \}$ (here, $\mathcal{S}_p$ is the image of $\mathcal{S}$ under the derived pullback along natural morphism on $\operatorname{Spec}(\mathcal{O}_{X,p}) \to X$).
\end{lemma}

\begin{proof}
    It is straightforward to see that $\operatorname{level}^{\mathcal{S}} (E) \geq {\sup}_{p \in X}\{\operatorname{level}^{\mathcal{S}_p} (E_p) \}$. 
    Hence, we only need to check the reverse inequality. If ${\sup}_{p \in X}\{\operatorname{level}^{\mathcal{S}_p} (E_p) \}=\infty$, then there nothing to show, so we can impose it is finite. By \cite[Lemma 5.5]{DeDeyn/Lank/ManaliRahul:2025}, we see for each $p\in X$ there is an $S(p)\in \mathcal{S}$ such that $N_p: = \operatorname{level}^{S(p)_p} (E_p) = \operatorname{level}^{\mathcal{S}_p} (E_p)$. Now, \cite[Proposition 3.5]{Letz:2021} tells us $U_p := \{ q\in X \mid \operatorname{level}^{S(p)_q} (E_q) \leq N_p \}$ is an open subset of $X$. As $p\in U_p$ this gives us an affine open cover of $X$. From $X$ being quasi-compact, we can refine this to a finite subcover, say $U_{p_1},\ldots,U_{p_n}$. Set $A:=\oplus^n_{i=1} S(p_i)$ and $N:={\sup}\{N_1,\dots,N_n\}$. Then \cite[Corollary 3.4]{Letz:2021} implies
    \begin{displaymath}
        \operatorname{level}^A (E) = \underset{p\in X}{\sup} \{ \operatorname{level}^{S_p} (E_p)\} \leq N \leq  \underset{p \in X}{\sup}\{\operatorname{level}^{\mathcal{S}_p} (E_p) \}.
    \end{displaymath}
    However, we know that $\operatorname{level}^{\mathcal{S}} (E)\leq \operatorname{level}^A (E)$, and so the claim follows.
\end{proof}

\begin{proposition}
    \label{prop:measurement_for_affine_local_to_global}
    Let $X=\operatorname{Spec}(R)$ where $R$ is a normal Noetherian integral domain. If $X$ admits a resolution of singularities, then $\mu_{\operatorname{bds}}(X)=\sup_{p\in X} \{ \mu_{\operatorname{bds}}(\mathcal{O}_{X,p}) \}$.
\end{proposition}

\begin{proof}
    By \Cref{cor:localization_bds}, it follows that $\sup_{p\in X} \{ \mu_{\operatorname{bds}}(\mathcal{O}_{X,p}) \}\leq \mu_{\operatorname{bds}}(X)$. Hence, we only need to show that $\mu_{\operatorname{bds}}(X)\leq \sup_{p\in X} \{ \mu_{\operatorname{bds}}(\mathcal{O}_{X,p}) \}$. 
    Let $f\colon \widetilde{X}\to X$ be a resolution of singularities. Consider the fibered square
    \begin{displaymath}
        \begin{tikzcd}
            {\widetilde{X}\times_X \operatorname{Spec}(\mathcal{O}_{X,p})} & {\operatorname{Spec}(\mathcal{O}_{X,p})} \\
            {\widetilde{X}} & {X}\rlap{ .}
            \arrow["{f^\prime}", from=1-1, to=1-2]
            \arrow["{s^\prime}"', from=1-1, to=2-1]
            \arrow["s", from=1-2, to=2-2]
            \arrow["f"', from=2-1, to=2-2]
        \end{tikzcd}
    \end{displaymath}
    Note $f^\prime$ is a resolution of singularities.  
    Moreover, \Cref{lem:base_change_essentially_surjective} tells us $(s^\prime)^\ast \colon D^b_{\operatorname{coh}}(\widetilde{X}\times_X \operatorname{Spec}(\mathcal{O}_{X,p})) \to D^b_{\operatorname{coh}}(\mathcal{O}_{X,p})$ is essentially surjective. Then, using flat base change, we have that $ (\mathbf{R}f_\ast D^b_{\operatorname{coh}}(\widetilde{X}))_p = \mathbf{R}f^\prime_\ast D^b_{\operatorname{coh}}(\widetilde{X}\times_X \operatorname{Spec}(\mathcal{O}_{X,p}))$. From \Cref{thm:r_value_for_modification_finite}, we see that
    \begin{displaymath}
        \begin{aligned}
            \mu_{\operatorname{bds}}(\mathcal{O}_{X,p}) 
            &= \operatorname{level}^{\mathbf{R}f^\prime_\ast D^b_{\operatorname{coh}}(\widetilde{X}\times_X \operatorname{Spec}(\mathcal{O}_{X,p}))} (\mathcal{O}_{X,p}) 
            \\&= \operatorname{level}^{(\mathbf{R}f_\ast D^b_{\operatorname{coh}}(\widetilde{X}))_p} (\mathcal{O}_{X,p}).
        \end{aligned}
    \end{displaymath}
    Therefore, the desired claim follows by \Cref{thm:r_value_for_modification_finite} and \Cref{lem:generalized_Letz_level}.
\end{proof}

\section{Some consequences}
\label{sec:some_consequences}

In the spirit of \cite[Lemma 5.1]{Krah/Vial:2023}, we have the following criterion to descending birational derived splinters.

\begin{proposition}
    [Descending birational derived splinters]
    \label{prop:field_base_change_birational_derived_splinter}
    Let $s\colon Y\to X$ be a flat morphism of Noetherian schemes such that $\mathcal{O}_X \xrightarrow{ntrl.} \mathbf{R}s_\ast \mathcal{O}_Y$ splits. If $Y$ is a birational derived splinter, then so is $X$.
\end{proposition}

\begin{proof}
    Let $f\colon X^\prime \to X$ be a proper birational morphism. Consider the fibered square
    \begin{displaymath}
        \begin{tikzcd}
            {Y^\prime} & {X^\prime} \\
            Y & {X}\rlap{ .}
            \arrow["{s^\prime}", from=1-1, to=1-2]
            \arrow["{f^\prime}"', from=1-1, to=2-1]
            \arrow["f", from=1-2, to=2-2]
            \arrow["s"', from=2-1, to=2-2]
        \end{tikzcd}
    \end{displaymath}
    Note that $f^\prime$ must be proper birational. So, from the hypothesis that $Y$ is a birational derived splinter, we have $\mathcal{O}_Y \xrightarrow{ntrl.} \mathbf{R}f^\prime_\ast \mathcal{O}_{Y^\prime}$ splits. 
    Applying $\mathbf{R}s_\ast$ and using that $\mathcal{O}_X \xrightarrow{ntrl.} \mathbf{R}s_\ast \mathcal{O}_Y$ splits, we get that $\mathcal{O}_X \xrightarrow{ntrl.} \mathbf{R}s_\ast \mathbf{R}f^\prime_\ast \mathcal{O}_{Y^\prime}$ splits. Since $\mathcal{O}_X \xrightarrow{ntrl.} \mathbf{R}f_\ast \mathcal{O}_{X^\prime}$ factors through this latter map, it must also split (see e.g.\ \Cref{lem:splitting_lemma_compositions}). 
\end{proof}

\begin{example}
    [Birational derived splinter need not ascend]
    \label{ex:bad_behavior_2}
    There are two examples. We start with a very easy one. Take any purely inseparable extension of fields $L/k$. Then $L$ is a birational derived splinter, whereas $L \otimes_k L$ is not as it is not reduced.

    Next, we give another. Let $k:= \mathbb{F}_3 (t)$ where $t$ is transcendental. Consider the projective plane curve $C$ over $k$ given by the equation $y^2 z + x^3 - t z^3 = 0$. This is a quasi-compact regular scheme as it is normal. However, if we consider the base change $C^\prime$ of $C$ along the field extension $\mathbb{F}_3 (t^{\frac{1}{3}})$, then $C^\prime$ is a singular projective curve (see e.g.\ \cite[Remark 16]{Kollar:2011}). 
    By \Cref{lem:LV25}, $C$ is a birational derived splinter, whereas $C^\prime$ fails to be such. Indeed, \cite[Theorem B]{Lank/Venkatesh:2025} would imply $C^\prime$ is normal had it been a birational derived splinter, which is absurd as it is singular.
\end{example}

With \Cref{ex:bad_behavior_2} in mind, we work over perfect fields so that regular schemes are smooth over the base (see e.g.\ \cite[\href{https://stacks.math.columbia.edu/tag/0B8X}{Tag 0B8X}]{StacksProject}), and so the latter property is stable under base change.

\begin{proposition}
    \label{prop:independence_field}
    Let $X$ be a scheme of finite type over a perfect field $k$. Assume there is a resolution of singularities $f\colon \widetilde{X}\to X$. Then $\mu_{\operatorname{bds}}(X) \geq \mu_{\operatorname{bds}}(X\times_k \operatorname{Spec}(L))$ for every field extension $L/k$. Additionally, if $L/k$ is finite, then $\mu_{\operatorname{bds}}(X) = \mu_{\operatorname{bds}}(X\times_k \operatorname{Spec}(L))$.
\end{proposition}

\begin{proof}
    We start with a setup/notation. Let $L/k$ be a (resp.\ finite) field extension. Note $\widetilde{X}$ must be smooth over $k$. There is a commutative diagram
    \begin{displaymath}
        \begin{tikzcd}
            && {\llap{$\widetilde{X}_L:=\;$}\widetilde{X}\times_k \operatorname{Spec}(L)} && {\widetilde{X}} \\
            {\llap{$X_L:=\;$}X\times_k \operatorname{Spec}(L)} && X \\
            {\operatorname{Spec}(L)} && {\operatorname{Spec}(k)}
            \arrow["{\pi^\prime}", from=1-3, to=1-5]
            \arrow["{f^\prime}"', bend right= 20, from=1-3, to=2-1, start anchor={[xshift=-2.4em,yshift=-2.5px]}]
            \arrow["{p^\prime}"{description, pos=0.8}, dashed, from=1-3, to=3-1]
            \arrow["f"', from=1-5, to=2-3]
            \arrow["p", bend left=20, from=1-5, to=3-3]
            \arrow["\pi"', crossing over, from=2-1, to=2-3]
            \arrow[from=2-1, to=3-1]
            \arrow["q"', from=2-3, to=3-3]
            \arrow[from=3-1, to=3-3]
        \end{tikzcd}
    \end{displaymath}
    whose square faces are fibered. Note that $\pi$ and $\pi^\prime$ are (resp.\ finite) flat morphisms by base change along $L/k$. Also, $f^\prime$ is a resolution of singularities.

    Suppose $\mathcal{O}_X\in \langle \mathbf{R}f_\ast D^b_{\operatorname{coh}}(\widetilde{X}) \rangle_n$ for some $n\geq 1$. Then
    \begin{displaymath}
        \mathcal{O}_{X_L}= \pi^\ast \mathcal{O}_X\in \langle \pi^\ast \mathbf{R}f_\ast D^b_{\operatorname{coh}}(\widetilde{X}) \rangle_n.
    \end{displaymath}
    By flat base change, we see that 
    \begin{displaymath}
        \begin{aligned}
            \mathcal{O}_{X_L}
            & \in \langle \pi^\ast \mathbf{R}f_\ast D^b_{\operatorname{coh}}(\widetilde{X}) \rangle_n
            \\&\subseteq \langle \mathbf{R}f^\prime_\ast (\pi^\prime)^\ast D^b_{\operatorname{coh}}(X_L) \rangle_n
            \\&\subseteq \langle \mathbf{R}f^\prime_\ast  D^b_{\operatorname{coh}}(\widetilde{X}_L) \rangle_n.
        \end{aligned}
    \end{displaymath}
    Hence, if coupled with \Cref{thm:r_value_for_modification_finite}, one has $\mu_{\operatorname{bds}}(X_L)\leq \mu_{\operatorname{bds}}(X)$.

    Next, we impose $L/k$ be a finite field extension. Let $n\geq 1$ such that
    \begin{displaymath}
        \mathcal{O}_{X_L}\in \langle \mathbf{R}f^\prime_\ast D^b_{\operatorname{coh}}(\widetilde{X}_L) \rangle_n.
    \end{displaymath}
    By flat base change with the inner most square of the diagram above, we have an isomorphism of bounded complexes $\pi_\ast \mathcal{O}_{X_L} \cong \mathcal{O}_X^{\oplus r_0}$. Now, tensoring with any object $E\in D^b_{\operatorname{coh}}(X)$ and using projection formula, we can find a section $E \to \pi_\ast \pi^\ast E$. Clearly, as $\pi$ is flat, $\pi^\ast E\in D^b_{\operatorname{coh}}(X_L)$. Hence, we see that $\pi_\ast \colon D^b_{\operatorname{coh}}(X_L)\to D^b_{\operatorname{coh}}(X)$ is essentially dense. A similar argument shows the same thing for $\pi^\prime_\ast \colon D^b_{\operatorname{coh}}(\widetilde{X}_L) \to D^b_{\operatorname{coh}}(\widetilde{X})$. Then we have
    \begin{displaymath}
        \begin{aligned}
            \mathcal{O}_X &\in \langle \pi_\ast \mathcal{O}_{X_L} \rangle_1 
            \\&\subseteq \langle \pi_\ast \mathbf{R}f^\prime_\ast D^b_{\operatorname{coh}}(\widetilde{X}_L) \rangle_n
            \\&\subseteq \langle \mathbf{R}f_\ast \pi^\prime_\ast  D^b_{\operatorname{coh}}(\widetilde{X}_L) \rangle_n
            \\&\subseteq \langle \mathbf{R}f_\ast D^b_{\operatorname{coh}}(\widetilde{X}) \rangle_n.
        \end{aligned}
    \end{displaymath}
    Hence, $\mu_{\operatorname{bds}}(X) \leq \mu_{\operatorname{bds}}(X_L)$, which completes the proof.
\end{proof}

\begin{theorem}
    \label{thm:base_change_bds_field}
    Let $Y$ be a scheme of finite type over a perfect field $k$ which admits a resolution of singularities. Then the following are equivalent:
    \begin{enumerate}
        \item \label{thm:base_change_bds_field1} $Y$ is a birational derived splinter
        \item \label{thm:base_change_bds_field2} $\operatorname{Spec}(L)\times_k Y$ is a birational derived splinter for some field extension $L/k$
        \item \label{thm:base_change_bds_field3} $\operatorname{Spec}(L)\times_k Y$ is a birational derived splinter for every field extension $L/k$.
    \end{enumerate}
\end{theorem}

\begin{proof}
    This follows from \Cref{prop:field_base_change_birational_derived_splinter} and \Cref{prop:independence_field}
\end{proof}

\begin{proposition}
    \label{prop:submultiplicativity}
    Let $X$ and $Y$ be schemes which are proper over a field $k$. Then
    \begin{displaymath}
        \max\{\mu_{\operatorname{bds}}(X),\mu_{\operatorname{bds}}(Y)\}\leq \mu_{\operatorname{bds}}(X\times_k Y).
    \end{displaymath} 
\end{proposition}

\begin{proof}
    We only show the inequality for $X$ and $X\times_k Y$ as the same argument applies for $Y$. Let $f\colon X^\prime \to X$ be a proper birational morphism. Consider the commutative diagram
    \begin{displaymath}
        \begin{tikzcd}
            {X^\prime\times_k Y} & {X\times_k Y} & Y \\
            {X^\prime} & X & {\operatorname{Spec}(k)}
            \arrow["{f^\prime}", from=1-1, to=1-2]
            \arrow["{\pi^\prime_1}"', from=1-1, to=2-1]
            \arrow["{\pi_2}"', from=1-2, to=1-3]
            \arrow["{\pi_1}", from=1-2, to=2-2]
            \arrow[from=1-3, to=2-3]
            \arrow["f", from=2-1, to=2-2]
            \arrow[from=2-2, to=2-3]
        \end{tikzcd}
    \end{displaymath}
    whose squares are fibered. By flat base change, $\mathbf{R}(\pi_1)_\ast \mathcal{O}_{X\times_k Y}$ and $\mathbf{R}(\pi^\prime_1)_\ast \mathcal{O}_{X^\prime \times_k Y}$ are isomorphic to bounded complexes of the form $\oplus_{t\in \mathbb{Z}} \mathcal{O}_X^{\oplus s_t}[t]$ and $\oplus_{t\in \mathbb{Z}} \mathcal{O}_{X^\prime}^{\oplus s_t}[t]$ respectively.
    Then \Cref{lem:thick_one_iff_splitting_naturally} implies $\langle \mathbf{R}(\pi_1)_\ast D^b_{\operatorname{coh}}(X\times_k Y) \rangle_1 = D^b_{\operatorname{coh}}(X)$ and $\langle \mathbf{R}(\pi_1^\prime)_\ast D^b_{\operatorname{coh}}(X^\prime\times_k Y) \rangle_1 = D^b_{\operatorname{coh}}(X^\prime)$. Note $f^\prime$ is proper and birational. Choose $n\leq \mu_{\operatorname{bds}}(X\times_k Y)$ such that $\mathcal{O}_{X^\prime \times_k Y}\in \langle \mathbf{R}f^\prime_\ast D^b_{\operatorname{coh}}(X\times_k Y) \rangle_n$. Then we have a string of inclusions
    \begin{displaymath}
        \begin{aligned}
            \mathcal{O}_X 
            &\in \langle \mathbf{R}(\pi_1)_\ast \mathcal{O}_{X\times_k Y}\rangle_1 
            \\&\subseteq \langle \mathbf{R}(\pi_1 \circ f^\prime)_\ast D^b_{\operatorname{coh}}(X^\prime \times_k Y) \rangle_n
            \\&\subseteq \langle \mathbf{R}(f\circ \pi_1^\prime)_\ast D^b_{\operatorname{coh}}(X^\prime \times_k Y) \rangle_n
            \\&\subseteq \langle \mathbf{R}f_\ast D^b_{\operatorname{coh}}(X) \rangle_n.
        \end{aligned}
    \end{displaymath}
    showing the desired inequality.
\end{proof}

\begin{theorem}
    \label{thm:submultiplicativity}
    Let $X$ and $Y$ be schemes which are proper over a perfect field $k$. Assume both admit resolutions of singularities. Then there are inequalities:
    \begin{displaymath}
        \max\{\mu_{\operatorname{bds}}(X),\mu_{\operatorname{bds}}(Y)\}\leq \mu_{\operatorname{bds}}(X\times_k Y) \leq \mu_{\operatorname{bds}}(X) \mu_{\operatorname{bds}}(Y).
    \end{displaymath}
\end{theorem}

\begin{proof} 
    Let $f\colon \widetilde{X}\to X$ and $g\colon \widetilde{Y}\to Y$ be resolutions of singularities. There is a commutative diagram
    \begin{displaymath}
        \begin{tikzcd}
            {\widetilde{X}\times_k \widetilde{Y}} & {X\times_k \widetilde{Y}} & {\widetilde{Y}} \\
            {\widetilde{X}\times_k Y} & {X\times_k Y} & Y \\
            {\widetilde{X}} & X & {\operatorname{Spec}(k)}
            \arrow["{f^{\prime \prime}}"', from=1-1, to=1-2]
            \arrow["{g^{\prime \prime}}", from=1-1, to=2-1]
            \arrow[from=1-2, to=1-3]
            \arrow["{g^\prime}", from=1-2, to=2-2]
            \arrow["g", from=1-3, to=2-3]
            \arrow["{f^\prime}"', from=2-1, to=2-2]
            \arrow["h", from=2-1, to=3-1]
            \arrow["{\pi_2}", from=2-2, to=2-3]
            \arrow["{\pi_1}"', from=2-2, to=3-2]
            \arrow[from=2-3, to=3-3]
            \arrow["f"', from=3-1, to=3-2]
            \arrow[from=3-2, to=3-3]
        \end{tikzcd}
    \end{displaymath}
    whose squares are all fibered. Note that $\pi_1$, $\pi_1 \circ g^\prime$, $\pi_2$ and $\pi_2 \circ f^\prime$ are flat morphisms by base change. Then  $f^\prime$, $f^{\prime \prime}$, $g^\prime$, and $g^{\prime \prime}$ are proper birational morphisms. 
    So, $f^\prime \circ g^{\prime \prime}$ is a proper birational morphism to $X\times_k Y$ from a smooth scheme.

    Now, we show the desired upper bound for $\mu_{\operatorname{bds}}(X\times_k Y)$. Let $n_1,n_2\geq 1$ such that $\mathcal{O}_X \in \langle \mathbf{R}f_\ast D^b_{\operatorname{coh}}(\widetilde{X}) \rangle_{n_1}$ and $\mathcal{O}_Y \in \langle \mathbf{R}g_\ast D^b_{\operatorname{coh}}(\widetilde{Y}) \rangle_{n_2}$. As $\widetilde{X}$ is smooth, and hence regular, we have that $D^b_{\operatorname{coh}}(\widetilde{X}) = \operatorname{Perf}(\widetilde{X})$. By flat base change, we see that 
    \begin{displaymath}
        \begin{aligned}
            \mathcal{O}_{X\times_k Y} 
            &\in \langle \pi^\ast_1 \mathbf{R}f_\ast D^b_{\operatorname{coh}}(\widetilde{X}) \rangle_{n_1}
            \\&\subseteq \langle \mathbf{R}f^\prime_\ast  \mathbf{L}h^\ast D^b_{\operatorname{coh}}(\widetilde{X}) \rangle_{n_1}
            \\&\subseteq \langle \mathbf{R}f^\prime_\ast \operatorname{Perf}(\widetilde{X}\times_k Y) \rangle_{n_1}.
        \end{aligned}
    \end{displaymath}
    From \Cref{prop:subadditive}, we know that $\mu_{\operatorname{bds}}(\widetilde{X}\times_k Y) \leq \mu_{\operatorname{bds}}(Y)$. Hence,
    \begin{displaymath}
        \mathcal{O}_{\widetilde{X}\times_k Y} \in \langle \mathbf{R}g_\ast^{\prime \prime} D^b_{\operatorname{coh}}(\widetilde{X}\times_k \widetilde{Y}) \rangle_{n_2},
    \end{displaymath}
    and so, by \Cref{lem:characterize_by_perfects},
    \begin{displaymath}
        \operatorname{Perf}(\widetilde{X}\times_k Y) \subseteq \langle \mathbf{R}g_\ast^{\prime \prime} D^b_{\operatorname{coh}}(\widetilde{X}\times_k \widetilde{Y}) \rangle_{n_2}.
    \end{displaymath}
    Combining this with the previous inclusion, we get
    \begin{displaymath}
        \begin{aligned}
            \mathcal{O}_{X\times_k Y}
            &\in \langle \mathbf{R}f^\prime_\ast \operatorname{Perf}(\widetilde{X}\times_k Y) \rangle_{n_1}
            \\&\subseteq \langle \mathbf{R}(f^\prime \circ g^{\prime \prime})_\ast D^b_{\operatorname{coh}}(\widetilde{X}\times_k \widetilde{Y}) \rangle_{n_1 n_2}.
        \end{aligned}
    \end{displaymath}
    From \Cref{thm:r_value_for_modification_finite}, we obtain the desired upper bound for $\mu_{\operatorname{bds}}(X\times_k Y)$. As the desired lower bound follows from \Cref{prop:submultiplicativity}, we are done.
\end{proof}

\begin{corollary}
    \label{cor:fibered_product_bds}
    Let $Y_1$ and $Y_2$ be schemes which are proper over a perfect field $k$. Assume both admit proper birational morphisms from regular schemes. Then $Y_1$ and $Y_2$ both are birational derived splinters if, and only if, $Y_1\times_k Y_2$ is such.
\end{corollary}

\begin{proof}
    This is immediate from \Cref{thm:submultiplicativity}.
\end{proof}

\bibliographystyle{alpha}
\bibliography{mainbib}

\end{document}